\documentclass[10pt]{article}

\usepackage{amsmath,amsbsy,amssymb,latexsym,amsthm,dsfont}
\usepackage{a4wide}

\usepackage{cite}

\newtheorem{thm}{Theorem}[section]

\newtheorem{lem}[thm]{Lemma}

\newtheorem{defn}[thm]{Definition}
\newtheorem{rem}[thm]{Remark}
\newtheorem{ex}[thm]{Example}

\newenvironment{keywords}{\begin{center}
\begin{minipage}[c]{13.4cm} {\bf Keywords:}} {\end{minipage}
\end{center}}

\begin{document}

\title{Fractional variational calculus 
for nondifferentiable functions\thanks{Submitted 13-Aug-2010; 
revised 24-Nov-2010; accepted 28-March-2011; for publication in 
\emph{Computers and Mathematics with Applications}.}}

\author{Ricardo Almeida\\
\texttt{ricardo.almeida@ua.pt}
\and Delfim F. M. Torres\\
\texttt{delfim@ua.pt}}

\date{Department of Mathematics, University of Aveiro,\\
Campus Universit\'{a}rio de Santiago, 3810-193 Aveiro, Portugal}

\maketitle

% ----------------------------------------

\begin{abstract}
We prove necessary optimality conditions,
in the class of continuous functions,
for variational problems defined with Jumarie's
modified Riemann--Liouville derivative.
The fractional basic problem of the calculus of variations
with free boundary conditions is considered,
as well as problems with isoperimetric and holonomic constraints.
\end{abstract}

\begin{keywords}
fractional calculus,
Jumarie's modified Riemann--Liouville derivative,
natural boundary conditions,
isoperimetric problems,
holonomic constraints.
\end{keywords}

\maketitle

% -----------------------------------------------------------

\section{Introduction}

There exists a vast literature on different definitions of fractional derivatives.
The most popular ones are the Riemann--Liouville and the Caputo derivatives.
Each fractional derivative presents some advantages and disadvantages
(see, \textrm{e.g.}, \cite{Miller,Podlubny,samko}). The Riemann--Liouville
derivative of a constant is not zero while Caputo's derivative of a constant
is zero but demands higher conditions of regularity for differentiability:
to compute the fractional derivative of a function in the Caputo sense,
we must first calculate its derivative. Caputo derivatives
are defined only for differentiable functions while
functions that have no first order derivative might have fractional
derivatives of all orders less than one
in the Riemann--Liouville sense \cite{Ross:Samko:Love}.

Recently, Guy Jumarie
(see \cite{Jumarie1,Jumarie2,Jumarie:2007,Jumarie3,Jumarie:AML2,Jumarie:2010,Jumarie8})
proposed a simple alternative definition to the Riemann--Liouville derivative.
His modified Riemann--Liouville derivative has the advantages
of both the standard Riemann--Liouville and Caputo
fractional derivatives: it is defined for arbitrary continuous
(nondifferentiable) functions and the fractional derivative
of a constant is equal to zero. Here we show that Jumarie's derivative
is more advantageous for a general theory of the calculus of variations.

The fractional calculus of variations is a recent research area
much in progress. It is being mainly developed
for Riemann--Liouville (see, \textrm{e.g.},
\cite{AGRA1,AlmeidaTorres,MyID:154,Baleanu,frederico,MyID:181})
and Caputo derivatives (see, \textrm{e.g.},
\cite{agrawalCap,MyID:147,Baleanu:Agrawal,MyID:149,MyID:169,MyID:163}).
For more on the calculus of variations,
in terms of other fractional derivatives, we refer the reader
to \cite{MyID:152,MyID:179,Cresson,El-Nabulsi:Torres,Klimek}
and references therein.

As pointed out in \cite{Ata:et:al},
the fractional calculus of variations in
Riemann--Liouville sense, as it is known,
has some problems, and results should be used with care.
Indeed, in order for the Riemann--Liouville derivatives
${_aD_x^\alpha}y(x)$ and  ${_xD_b^\alpha}y(x)$ to be continuous
on a closed interval $[a,b]$, the boundary conditions $y(a)=0$ and $y(b)=0$
must be satisfied \cite{Ross:Samko:Love}. This is very restrictive when working
with variational problems of minimizing or maximizing functionals subject to arbitrarily
given boundary conditions, as often done in the calculus of variations
(see Proposition 1 and Remark 2 of \cite{MyID:154}).
With Jumarie's fractional derivative this situation does not occur,
and one can consider general boundary conditions $y(a)=y_a$ and $y(b)=y_b$.

The paper is organized as follows. In Section~\ref{sec:pre}
we state the assumptions, notations, and
the results of the literature needed in the sequel.
Section~\ref{sec3} reviews Jumarie's fractional
Euler--Lagrange equations \cite{Jumarie5}.
Our contribution is then given in Section~\ref{sec4}:
in \S\ref{sec5} we consider the case when no boundary conditions
are imposed on the problem, and we prove associated transversality
(natural boundary) conditions; optimization with constraints
(integral or not) are studied in sections \S\ref{sec6} and \S\ref{sec7}.
Finally, in Section~\ref{sec:Because_of_Reviewer2} we explain the novelties
of our results with respect to previous results in the literature.

% ----------------------------------------

\section{Preliminaries on Jumarie's Riemann--Liouville derivative}
\label{sec:pre}

Throughout the text $f:[0,1]\to\mathbb R$ is a continuous function
and $\alpha$ a real number on the interval $(0,1)$.
Jumarie's modified Riemann--Liouville fractional derivative is defined by
$$
f^{(\alpha)}(x)
=\frac{1}{\Gamma(1-\alpha)}\frac{d}{dx}
\int_0^x(x-t)^{-\alpha}(f(t)-f(0))\,dt.
$$
If $f(0)=0$, then $f^{(\alpha)}$ is equal to the Riemann--Liouville
fractional derivative of $f$ of order $\alpha$. We remark that the
fractional derivative of a constant is zero, as desired. Moreover,
$f(0)=0$ is no longer a necessary condition for the fractional
derivative of $f$ to be continuous on $[0,1]$.

The $(dt)^\alpha$ integral of $f$ is defined as follows:

$$\int_0^xf(t)(dt)^\alpha=\alpha\int_0^x(x-t)^{\alpha-1}f(t)dt.$$
For a motivation of this definition, we refer to \cite{Jumarie1}.

\begin{rem}
This type of fractional derivative and integral
has found applications in some physical phenomena.
The definition of the fractional derivative via difference reads
$$
f^{(\alpha)}(x)=\lim_{h\downarrow0}\frac{\triangle ^\alpha f(x)}{h^\alpha},
\quad 0<\alpha<1 \, ,
$$
and obviously this contributes some questions on the sign of $h$,
as it is emphasized by the fractional Rolle's formula
$f(x+h) \cong f(x)+h^\alpha f^{(\alpha)}(x)$.
In a first approach, in a realm of physics, when $h$ denotes time,
then this feature could picture the irreversibility of time. The fractional derivative
is quite suitable to describe dynamics evolving in space which exhibit coarse-grained phenomenon.
When the point in this space is not infinitely thin but rather a thickness,
then it would be better to replace  $dx$ by $(dx)^\alpha$ , $0 <\alpha<1$,
where $\alpha$ characterizes the grade of the phenomenon. The fractal feature of the space
is transported on time, and so both space and time are fractal. Thus, the increment of time
of the dynamics of the system is not $dx$ but $(dx)^\alpha$. For more on the subject see,
\textrm{e.g.}, \cite{AlmeidaMalinowskaTorres,Filatova,Jumarie5,Jumarie6,Jumarie:2010}.
\end{rem}

Our results make use of the formula of integration by parts for the $(dx)^\alpha$ integral.
This formula follows from the fractional Leibniz rule and the fractional Barrow's formula.

\begin{thm}[Fractional Leibniz rule \cite{Jumarie4}]
\label{thm:FLR}
If $f$ and $g$ are two continuous functions on $[0,1]$, then
\begin{equation}
\label{leibniz}
(f(x)g(x))^{(\alpha)}=(f(x))^{(\alpha)}g(x)+f(x)(g(x))^{(\alpha)}.
\end{equation}
\end{thm}

Kolwankar obtained the same formula \eqref{leibniz}
by using an approach on Cantor space \cite{MR2109983}.

\begin{thm}[Fractional Barrow's formula \cite{Jumarie3}]
\label{thm:FBF}
For a continuous function $f$, we have
$$\int_0^xf^{(\alpha)}(t)(dt)^{(\alpha)}=\alpha! (f(x)-f(0)),$$
where $\alpha!=\Gamma(1+\alpha)$.
\end{thm}

From Theorems~\ref{thm:FLR} and \ref{thm:FBF}
we deduce the following formula of integration by parts:
\begin{equation*}
\begin{array}{ll}\displaystyle \int_0^1 u^{(\alpha)}(x)v(x)\, (dx)^\alpha
&=\displaystyle\int_0^1 (u(x)v(x))^{(\alpha)}\, (dx)^\alpha
-\int_0^1 u(x)v^{(\alpha)}(x)\, (dx)^\alpha\\
&=\displaystyle\alpha! [u(x)v(x)]_0^1-\int_0^1 u(x)v^{(\alpha)}(x)\, (dx)^\alpha.
\end{array}
\end{equation*}

It has been proved that the fractional Taylor series holds for nondifferentiable functions.
See, for instance, \cite{Jumarie7}. Another approach is to check that this formula holds
for the Mittag--Leffler function, and then to consider functions which can be approximated by the former.
The first term of this series is the Rolle's fractional formula which has been obtained
by Kolwankar and Jumarie and provides the equality
$d^\alpha x(t)=\alpha! dx(t)$.

It is a simple exercise to verify that the fundamental lemma of the calculus
of variations is valid for the $(dx)^\alpha$ integral
(see, \textrm{e.g.}, \cite{Brunt} for a standard proof):

\begin{lem}
\label{fundLemma}
Let $g$ be a continuous function and assume that
$$\int_0^1g(x)h(x)\, (dx)^\alpha=0$$
for every continuous function $h$ satisfying $h(0)=h(1)=0$. Then $g \equiv0$.
\end{lem}

% ----------------------------------------

\section{Jumarie's Euler--Lagrange equations}
\label{sec3}

Consider functionals
\begin{equation}
\label{funct}
\mathcal{J}(y)=\int_0^1
L\left(x,y(x),y^{(\alpha)}(x)\right) \,(dx)^\alpha
\end{equation}
defined on the set of continuous curves $y:[0,1]\to\mathbb R$,
where $L(\cdot,\cdot,\cdot)$ has continuous partial derivatives
with respect to the second and third variable. Jumarie
has addressed in \cite{Jumarie5} the basic problem
of calculus of variations: to minimize (or maximize)
$\mathcal{J}$, when restricted to the class of continuous curves satisfying
prescribed boundary conditions $y(0)=y_0$ and $y(1)=y_1$.
Let us denote this problem by $(P)$.
A necessary condition for problem $(P)$ is given by the next result.

\begin{thm}[The Jumarie fractional Euler--Lagrange equation \cite{Jumarie5}---scalar case]
\label{ELequationTheo}
If $y$ is a solution to the basic fractional problem
of the calculus of variations $(P)$, then
\begin{equation}
\label{ELequation}
\frac{\partial L}{\partial y}-\frac{d^\alpha}{dx^\alpha}
\frac{\partial L}{\partial y^{(\alpha)}} =0
\end{equation}
is satisfied along $y$, for all $x\in[0,1]$.
\end{thm}

\begin{defn}
A curve that satisfies equation \eqref{ELequation}
for all $x \in[0,1]$ is said to be an \emph{extremal}
for $\mathcal{J}$.
\end{defn}

\begin{ex}
\label{ex:1}
Consider the following problem:
$$
\mathcal{J}(y)=\int_0^1\left[
\frac{x^\alpha}{\Gamma(\alpha+1)}(y^{(\alpha)})^2
-2x^\alpha y^{(\alpha)} \right]^2\,(dx)^\alpha \longrightarrow  \textrm{extremize}
$$
subject to the boundary conditions
$$
y(0)=1 \mbox{ and } y(1)=2.
$$
The Euler--Lagrange equation associated to this problem is
\begin{equation}
\label{ELexample}
-\frac{d^\alpha}{dx^\alpha}\left(2\left[
\frac{x^\alpha}{\Gamma(\alpha+1)}(y^{(\alpha)})^2
-2x^\alpha y^{(\alpha)} \right]\cdot \left[
\frac{2x^\alpha}{\Gamma(\alpha+1)}y^{(\alpha)}
-2x^\alpha \right]\right)=0.
\end{equation}
Let $y=x^\alpha+1$. Since $y^{(\alpha)}=\Gamma(\alpha+1)$, it follows that $y$
is a solution of \eqref{ELexample}. We remark that the extremal curve
is not differentiable in $[0,1]$.
\end{ex}

\begin{rem}
When $\alpha=1$, we obtain the variational functional
$$
\mathcal{J}(y)=\int_0^1\left[ x (y')^2-2xy' \right]^2\,dx.
$$
It is easy to verify that $y=x+1$ satisfies the
(standard) Euler--Lagrange equation.
\end{rem}

We now present the Euler--Lagrange equation for functionals
containing several dependent variables.

\begin{thm}[The Jumarie fractional Euler--Lagrange equation \cite{Jumarie5}---vector case]
\label{theo}
Consider a functional $\mathcal{J}$, defined on the set of curves
satisfying the boundary conditions $\textbf{y}(0)=\textbf{y}_0$
and $\textbf{y}(1)=\textbf{y}_1$, of the form
$$
\mathcal{J}(\textbf{y})
=\int_0^1 L\left(x,\textbf{y}(x),\textbf{y}^{(\alpha)}(x)\right)
\,(dx)^\alpha,
$$
where $\textbf{y}=(y_1,\ldots,y_n)$,
$\textbf{y}^{(\alpha)}=(y_1^{(\alpha)},\ldots,y_n^{(\alpha)})$,
and $y_k$, $k=1,\ldots,n$, are continuous real valued functions defined on $[0,1]$.
Let $\textbf{y}$ be an extremizer of $\mathcal{J}$. Then,
$$
\frac{\partial L}{\partial y_k}\left(x,\textbf{y}(x),\textbf{y}^{(\alpha)}(x)\right)
-\frac{d^\alpha}{dx^\alpha} \frac{\partial L}{\partial y_k^{(\alpha)}}\left(x,
\textbf{y}(x),\textbf{y}^{(\alpha)}(x)\right)=0,
$$
$k=1,\ldots,n$, for all $x\in[0,1]$.
\end{thm}

% ------------------------------------------------

\section{Main results}
\label{sec4}

We give new necessary optimality conditions for:
(i) functionals of form \eqref{funct}
with free boundary conditions (Theorem~\ref{thm:bc});
(ii) fractional isoperimetric problems of Jumarie (Theorem~\ref{thm:iso});
and (iii) Jumarie-type problems with subsidiary holonomic constraints
(Theorems~\ref{thm:hol} and \ref{thm:hol:gf}).

% ------------------

\subsection{Natural boundary conditions}
\label{sec5}

The problem is stated as follows. Given a functional
$$
\mathcal{J}(y)=\int_0^1L\left(x,y(x),y^{(\alpha)}(x)\right) \,(dx)^\alpha\, ,
$$
where the Lagrangian $L(\cdot,\cdot,\cdot)$ has continuous partial derivatives
with respect to the second and third variables, determine continuous curves
$y:[0,1]\to\mathbb R$ such that $\mathcal{J}$ has an extremum at $y$.
Note that no boundary conditions are now imposed.

\begin{thm}
\label{thm:bc}
Let $y$ be an extremizer for $\mathcal{J}$.
Then $y$ satisfies the Euler--Lagrange equation
\begin{equation}
\label{our:EL}
\frac{\partial L}{\partial y}
-\frac{d^\alpha}{dx^\alpha} \frac{\partial L}{\partial y^{(\alpha)}}=0
\end{equation}
for all $x\in[0,1]$ and the natural boundary conditions
\begin{equation}
\label{our:NBC}
\left. \frac{\partial L}{\partial y^{(\alpha)}}\right|_{x=0}=0
\quad \mbox{and} \quad \left. \frac{\partial L}{\partial y^{(\alpha)}}\right|_{x=1}=0.
\end{equation}
\end{thm}

\begin{proof}
Let $h$ be any continuous curve and let
$j(\epsilon)=\mathcal{J}(y+\epsilon h)$. It follows that
$$
\begin{array}{ll}
0 & = \displaystyle\int_0^1\left(\frac{\partial L}{\partial y}\cdot h(x)
+ \frac{\partial L}{\partial y^{(\alpha)}} \cdot h^{(\alpha)}(x) \right)\,(dx)^\alpha\\
&\\
&  = \displaystyle\int_0^1\left(  \frac{\partial L}{\partial y}
-\frac{d^\alpha}{dx^\alpha} \frac{\partial L}{\partial y^{(\alpha)}} \right)\cdot h(x)
\,(dx)^\alpha+\alpha! \left[\frac{\partial L}{\partial y^{(\alpha)}}h(x)\right]_{x=0}^{x=1}.
\end{array}
$$
If we choose curves such that $h(0)=h(1)=0$, we deduce by Lemma~\ref{fundLemma}
the Euler--Lagrange equation \eqref{our:EL}. Then condition
$$
\left[\frac{\partial L}{\partial y^{(\alpha)}}h(x)\right]_{x=0}^{x=1}=0
$$
must be verified. Picking curves such that $h(0)=0$ and $h(1)\not=0$,
and others such that $h(1)=0$ and $h(0)\not=0$,
we deduce the natural boundary conditions \eqref{our:NBC}.
\end{proof}

If one of the endpoints is specified,
say $y(0)=y_0$, then the necessary conditions become
$$
\frac{\partial L}{\partial y}-\frac{d^\alpha}{dx^\alpha}
\frac{\partial L}{\partial y^{(\alpha)}}=0
$$
and
$$
\left. \frac{\partial L}{\partial y^{(\alpha)}}\right|_{x=1}=0.
$$

\begin{ex}
Let $\mathcal{J}$ be given by the expression
$$
\mathcal{J}(y)=\int_0^1\sqrt{1+y^{(\alpha)}(x)^2} \,(dx)^\alpha.
$$
The Euler--Lagrange equation  associated to this problem is
$$
\frac{d^\alpha}{dx^\alpha}\frac{y^{(\alpha)}(x)}{\sqrt{1+y^{(\alpha)}(x)^2}}=0
$$
and the natural boundary conditions are
$$
\left.\frac{y^{(\alpha)}(x)}{\sqrt{1+y^{(\alpha)}(x)^2}}\right|_{x=1}=0
\quad \mbox{and} \quad \left.\frac{y^{(\alpha)}(x)}{\sqrt{1+y^{(\alpha)}(x)^2}}\right|_{x=0}=0.
$$
Since $y^{(\alpha)}=0$ if $y$ is a constant function, we have
that any constant curve is a solution to this problem.
\end{ex}

% ---------------------------------------------

\subsection{The isoperimetric problem}
\label{sec6}

The study of isoperimetric problems is an important area inside
the calculus of variations. One wants to find the extremizers of a given functional,
when restricted to a prescribed integral constraint. Problems of this type have
found many applications in differential geometry, discrete and convex geometry,
probability, Banach space theory, and multiobjective optimization
(see \cite{Ric:Del:ISO,Ric:Del:ISO:nabla,Malina} and references therein).
We introduce the isoperimetric fractional problem as follows:
to maximize or minimize the functional
$$
\mathcal{J}(y)=\int_0^1L(x,y(x),y^{(\alpha)}(x)) \,(dx)^\alpha
$$
when restricted to the conditions
\begin{equation}
\label{Isoconst}
\mathcal{G}(y)=\int_0^1f(x,y(x),y^{(\alpha)}(x)) \,(dx)^\alpha=K,\quad K \in \mathbb R,
\end{equation}
and
\begin{equation}
\label{boundconst}
y(0)=y_0 \mbox{ and } y(1)=y_1.
\end{equation}
Similarly as before, we assume that $L(\cdot,\cdot,\cdot)$
and $f(\cdot,\cdot,\cdot)$ have continuous partial derivatives
with respect to the second and third variables.

\begin{thm}
\label{thm:iso}
Let $y$ be an extremizer of $\mathcal{J}$ restricted to the set of curves
that satisfy conditions \eqref{Isoconst} and \eqref{boundconst}.
If $y$ is not an extremal for $\mathcal{G}$, then there exists a constant
$\lambda$ such that the curve $y$ satisfies the equation
$$
\frac{\partial F}{\partial y}
-\frac{d^\alpha}{dx^\alpha} \frac{\partial F}{\partial y^{(\alpha)}}=0
$$
for all $x\in[0,1]$, where $F=L-\lambda f$.
\end{thm}

\begin{proof}
Consider a variation curve of $y$ with two parameters,
say $y(x)+\epsilon_1h_1(x)+\epsilon_2h_2(x)$, where $h_1$ and $h_2$
are two continuous curves satisfying $h_i(0)=h_i(1)=0, \, i=1,2$.
Consider now two new functions defined in an open neighborhood of zero:
$$
j(\epsilon_1,\epsilon_2)=\mathcal{J}(y+\epsilon_1h_1+\epsilon_2h_2)
\mbox{ and } g(\epsilon_1,\epsilon_2)=\mathcal{G}(y+\epsilon_1h_1+\epsilon_2h_2)-K.
$$
Thus, $g(0,0)=0$ and
$$
\left. \frac{\partial g}{\partial\epsilon_2} \right|_{(0,0)}
=\int_0^1\left(  \frac{\partial f}{\partial y}
-\frac{d^\alpha}{dx^\alpha} \frac{\partial f}{\partial y^{(\alpha)}} \right)
\cdot h_2(x) \,(dx)^\alpha.
$$
Since $y$ is not an extremal for $\mathcal{G}$, there must exist some curve
$h_2$ such that $\partial g / \partial \epsilon_2 (0,0)\not=0$.
By the implicit function theorem we may write $g(\epsilon_1,\epsilon_2(\epsilon_1))=0$
for some function $\epsilon_2(\cdot)$ defined in an open neighborhood of zero.
Because $(0,0)$ is an extremum of $j$ subject to the constraint $g(\epsilon_1,\epsilon_2)=0$,
and $\nabla g(0,0)\not=0$, by the Lagrange multiplier rule
there exists some real $\lambda$ verifying the equation
$$
\nabla (j-\lambda g)(0,0)=0.
$$
In particular,
$$
\begin{array}{ll}
0 & = \displaystyle \left. \frac{\partial }{\partial\epsilon_1}(j-\lambda g) \right|_{(0,0)}\\
&\\
& =\displaystyle \int_0^1\left[  \frac{\partial L}{\partial y}
-\frac{d^\alpha}{dx^\alpha} \frac{\partial L}{\partial y^{(\alpha)}}
-\lambda \left(  \frac{\partial f}{\partial y}
-\frac{d^\alpha}{dx^\alpha} \frac{\partial f}{\partial y^{(\alpha)}} \right) \right]
\cdot h_1(x) \,(dx)^\alpha.
\end{array}
$$
By the arbitrariness of $h_1$ and by Lemma~\ref{fundLemma}, we have
$$
\frac{\partial L}{\partial y}-\frac{d^\alpha}{dx^\alpha} \frac{\partial L}{\partial y^{(\alpha)}}
-\lambda \left(  \frac{\partial f}{\partial y}
-\frac{d^\alpha}{dx^\alpha} \frac{\partial f}{\partial y^{(\alpha)}} \right)=0.
$$
Introducing $F=L-\lambda f$ we deduce that $y$ satisfies the equation
$$
\frac{\partial F}{\partial y}-\frac{d^\alpha}{dx^\alpha}
\frac{\partial F}{\partial y^{(\alpha)}} =0
$$
for all $x\in[0,1]$.
\end{proof}

% -----------------

\subsection{Holonomic constraints}
\label{sec7}

In \S\ref{sec6} the subsidiary conditions that the functions must satisfy
are given by integral functionals. We now consider a different type of problem:
find functions $y_1$ and $y_2$ for which the functional
\begin{equation}
\label{funct2}
\mathcal{J}(y_1,y_2)=\int_0^1L\left(x,y_1(x),
y_2(x),y_1^{(\alpha)}(x),y_2^{(\alpha)}(x)\right) \,(dx)^\alpha
\end{equation}
has an extremum, where the admissible functions satisfy the boundary conditions
\begin{equation}
\label{boundconst2}
(y_1(0),y_2(0))=(y_1^0,y_2^0) \mbox{ and } (y_1(1),y_2(1))=(y_1^1,y_2^1),
\end{equation}
and the subsidiary holonomic condition
\begin{equation}
\label{subsconst} g(x,y_1(x),y_2(x))=0.
\end{equation}

\begin{thm}
\label{thm:hol}
Given a functional $\mathcal{J}$ as in \eqref{funct2},
defined on the set of curves that satisfy the boundary conditions
\eqref{boundconst2} and lie on the surface \eqref{subsconst},
let $(y_1,y_2)$ be an extremizer for $\mathcal{J}$.
If $\partial g / \partial y_2\not=0$ for all $x\in[0,1]$,
then there exists a continuous function $\lambda(x)$
such that $(y_1,y_2)$ satisfy the Euler--Lagrange equations
\begin{equation}
\label{ELequation2}
\frac{\partial F}{\partial y_k}-\frac{d^\alpha}{dx^\alpha}
\frac{\partial F}{\partial y_k^{(\alpha)}}=0\, ,
\end{equation}
$k=1,2$, for all $x\in[0,1]$, where $F=L-\lambda g$.
\end{thm}

\begin{proof}
Let $y=(y_1,y_2)$, $\epsilon$ be a real,
and $(\hat y_1(x),\hat y_2(x))=y(x)+\epsilon h(x)$ with $h(x)=(h_1(x),h_2(x))$
a continuous curve such that $h(0)=h(1)=(0,0)$. By the implicit function theorem
it is possible to solve the equation $g(x,\hat y_1(x),\hat y_2(x))=0$
for $h_2$, \textrm{i.e.}, to write $h_2=h_2(\epsilon,h_1)$.
Let $j(\epsilon)=\mathcal{J}(\hat y_1(x),\hat y_2(x))$. Then $j'(0)=0$ and so
\begin{equation}
\label{subsequation}
\begin{array}{ll}
0 & = \displaystyle \int_0^1\left(  \frac{\partial L}{\partial y_1}h_1(x)
+\frac{\partial L}{\partial y_1^{(\alpha)}}h_1^{(\alpha)}(x)
+ \frac{\partial L}{\partial y_2}h_2(x)
+\frac{\partial L}{\partial y_2^{(\alpha)}}h_2^{(\alpha)}(x)\right) \,(dx)^\alpha.\\
&\\
&=\displaystyle \int_0^1\left( \left( \frac{\partial L}{\partial y_1}
-\frac{d^\alpha}{dx^\alpha} \frac{\partial L}{\partial y_1^{(\alpha)}}\right)h_1(x)
+ \left( \frac{\partial L}{\partial y_2}-\frac{d^\alpha}{dx^\alpha}
\frac{\partial L}{\partial y_2^{(\alpha)}}\right)h_2(x)\right) \,(dx)^\alpha.
\end{array}
\end{equation}
On the other hand, we required that $(\hat y_1(x),\hat y_2(x))$ satisfy
the condition \eqref{subsconst}. This means that
$g(x,\hat y_1(x),\hat y_2(x))=0$, and so it follows that
$$
\begin{array}{ll}
0 & = \displaystyle \left. \frac{d}{d\epsilon} g(x,\hat y_1(x),\hat y_2(x))\right|_{\epsilon=0}\\
&\\
&=\displaystyle  \frac{\partial g}{\partial y_1}h_1(x) + \frac{\partial g}{\partial y_2}h_2(x).
\end{array}
$$
Thus, we may write $h_2$ in the form
$$
h_2(x)=-\frac{\frac{\partial g}{\partial y_1}}{\frac{\partial g}{\partial y_2}}h_1(x).
$$
Define $\lambda$ as follows:
\begin{equation}
\label{subslambda}
\lambda(x)=\frac{ \frac{\partial L}{\partial y_2}
-\frac{d^\alpha}{dx^\alpha} \frac{\partial L}{\partial y_2^{(\alpha)}}}{\frac{\partial g}{\partial y_2}}.
\end{equation}
Then, we can rewrite equation \eqref{subsequation} as
$$
0 = \displaystyle \int_0^1 \left( \frac{\partial L}{\partial y_1}
-\frac{d^\alpha}{dx^\alpha} \frac{\partial L}{\partial y_1^{(\alpha)}}
-\lambda(x) \frac{\partial g}{\partial y_1}\right)h_1(x)\,(dx)^\alpha.
$$
By Lemma~\ref{fundLemma}, and since $h_1$ is an arbitrary curve, we deduce that
\begin{equation}
\label{subsEL}
\frac{\partial L}{\partial y_1}-\frac{d^\alpha}{dx^\alpha}
\frac{\partial L}{\partial y_1^{(\alpha)}}-\lambda(x) \frac{\partial g}{\partial y_1}=0.
\end{equation}
Let $F=L-\lambda g$. Combining equations \eqref{subslambda} and \eqref{subsEL}
we obtain formula \eqref{ELequation2}. This completes the proof of the theorem.
\end{proof}

We now state (without proof) our previous result in its general form.

\begin{thm}
\label{thm:hol:gf}
Let $\mathcal{J}$ be given by
$$
\mathcal{J}(\textbf{y})=\int_0^1L(x,\textbf{y}(x),\textbf{y}^{(\alpha)}(x)) \,(dx)^\alpha,
$$
where  $\textbf{y}=(y_1,\ldots,y_n)$ and $\textbf{y}^{(\alpha)}=(y_1^{(\alpha)},\ldots,y_n^{(\alpha)})$,
such that $y_k$, $k=1,\ldots,n$, are continuous functions defined on the set of curves
that satisfy the boundary conditions $\textbf{y}(0)=\bf{y_0}$ and $\textbf{y}(1)
=\bf{y_1}$ and satisfy the constraint $g(x,\textbf{y})=0$. If $\textbf{y}$ is an
extremizer for $\mathcal{J}$, and if $\partial g / \partial y_n\not=0$ for all
$x\in[0,1]$, then there exists a continuous function $\lambda(x)$
such that $\textbf{y}$ satisfy the Euler--Lagrange equations
$$
\frac{\partial F}{\partial y_k}
-\frac{d^\alpha}{dx^\alpha} \frac{\partial F}{\partial y_k^{(\alpha)}}=0\, ,
\quad k=1,\ldots,n \, ,
$$
for all $x\in[0,1]$, where $F=L-\lambda g$.
\end{thm}

%-----------------------------------------------------------------

\section{Comparison with previous results in the literature}
\label{sec:Because_of_Reviewer2}

The fractional variational calculus, dealing with Jumarie's modified
Riemann--Liouville derivative, is still at the very beginning.
Results available in the literature reduce to those in \cite{Jumarie5}
(published in 2009) and \cite{AlmeidaMalinowskaTorres,withBasiaRachid1}
(published in 2010). In \cite{Jumarie5} the basic problem of the calculus
of variations with Jumarie's Riemann--Liouville derivative is stated
for the first time. The fractional Euler--Lagrange equation
is obtained under the presence of external forces,
and the implication of Jumarie's fractional calculus
in physical situations discussed in detail,
in particular Lagrangian mechanics of fractional order.
A fractional theory of the calculus of variations for multiple integrals,
in the sense of Jumarie, is proposed in \cite{AlmeidaMalinowskaTorres}:
fractional versions of the theorems of Green and Gauss,
multi-time fractional Euler--Lagrange equations,
and fractional natural boundary conditions are proved. As an application,
the fractional equation of motion of a vibrating string is investigated
\cite{AlmeidaMalinowskaTorres}. In \cite{withBasiaRachid1}, Euler--Lagrange
necessary optimality conditions for fractional problems of the calculus
of variations which are given by a composition of functionals,
within Jumarie's fractional calculus, are proved.
Optimality conditions for the product and the quotient
of Jumarie's fractional variational functionals are obtained
as particular cases \cite{withBasiaRachid1}.

In the present paper we develop further the theory initiated
by Jumarie in 2009 \cite{Jumarie5}, by considering other
fundamental problems of variational type than those treated before in
\cite{AlmeidaMalinowskaTorres,Jumarie5,withBasiaRachid1}.
First we consider the general form of optimality, when we do not impose
constraints on the boundaries, \textrm{i.e.}, when $y(0)$ and/or $y(1)$ are free
(Section~\ref{sec5}). If $y(0)$ is not specified, then the
transversality condition
$\displaystyle \frac{\partial L}{\partial y^{(\alpha)}}\left(0,y(0),y^{(\alpha)}(0)\right)=0$
complements Jumarie's Euler--Lagrange equation;
if $y(1)$ is not specified, then the
transversality condition
$\displaystyle\frac{\partial L}{\partial y^{(\alpha)}}\left(1,y(1),y^{(\alpha)}(1)\right)=0$
holds together with Jumarie's Euler--Lagrange equation.
We then study fractional variational problems via Jumarie's
modified Riemann--Liouville derivative under the presence
of certain constraints, thus limiting the space of continuous functions
in which we search the extremizers. Two types of problems
were considered here for the first time in the literature:
fractional isoperimetric problems where the constraint also contains
Jumarie's modified Riemann--Liouville derivative
(Section~\ref{sec6}); and variational problems of Jumarie type
subject to  holonomic constraints (Section~\ref{sec7}).
The results of the paper are valid for extremizers
which are not differentiable.

As future work, we plan to study
nondifferentiable fractional variational problems under
the presence of nonholonomic constraints, which depend
on Jumarie's modified Riemann--Liouville derivative,
and possible applications in physics and economics.

%-----------------------------------------------------------------

\subsection*{Acknowledgments}

Work supported by the
{\it Center for Research and Development in Mathematics and Applications} (CIDMA)
from the ``Funda\c{c}\~{a}o para a Ci\^{e}ncia e a Tecnologia'' (FCT),
cofinanced by the European Community Fund FEDER/POCI 2010.
The authors are very grateful to two referees
for valuable remarks and comments,
which improved the quality of the paper.

%-----------------------------------------------------------------

%-----------------------------------------------------------------

\end{document}